\documentclass[12pt]{amsart}
\usepackage{amscd,amssymb}
\usepackage{amsthm,amsmath,enumerate,amssymb}
\usepackage[matrix,arrow]{xy}
\usepackage{enumerate}
\usepackage{longtable}
\usepackage{comment}

\newtheorem*{example*}{Example}

\newtheorem{theorem}[equation]{Theorem}
\newtheorem{lemma}[equation]{Lemma}
\newtheorem{corollary}[equation]{Corollary}
\newtheorem{proposition}[equation]{Proposition}

\newtheorem*{conjecture*}{Conjecture}

\newtheorem*{question*}{Question}

\newtheorem*{problem*}{Problem}
\newtheorem*{theorem*}{Theorem}

\theoremstyle{remark}
\newtheorem{remark}[equation]{Remark}
\newtheorem*{remark*}{Remark}

\newcommand{\Ric}{{\operatorname{Ric}}}
\newcommand{\vol}{{\operatorname{vol}}}
\newcommand{\idd}{\sqrt{-1}\partial\bar{\partial}}
\newcommand{\osc}{{\operatorname{osc}}}
\newcommand{\tr}{{\operatorname{tr}}}
\newcommand{\K}{\text{K\"ahler }}

\makeatletter\@addtoreset{equation}{section} \makeatother

\usepackage{hyperref}
\hypersetup{colorlinks,linkcolor={blue},citecolor={red}}

\title{Some refinements of the partial $C^0$ estimate}

\begin{document}
\pagestyle{headings}

\author[K.~Zhang]{Kewei Zhang}
\address{Beijing International Center for Mathematical Research, Peking University.}
\email{kwzhang@pku.edu.cn}

\maketitle

\begin{abstract}
    Relying on the recent work of Liu-Sz\'ekelyhidi we give a weak asymptotic estimate for the Bergman kernels of polarized K\"ahler manifolds with Ricci lower bound and Sobolev constant upper bound. We will also give a simple proof for the partial $C^0$ estimate along the (generalized) K\"ahler-Ricci flow on Fano manifolds.
\end{abstract}

\section{Introduction}
The goal of this paper is to give some refinements and applications on the topic of partial $C^0$ estimates for polarized K\"ahler manifolds.

\subsection{Background}
\hfill

Thoughout this paper, we denote by $(X,\omega,L,h)$ a polarized K\"ahler manifold, where $X$ is an $n$-dimensional ($n\geq2$) compact K\"ahler manifold, $L$ is an ample line bundle on $X$, $\omega$ is a K\"ahler form in the class $2\pi c_1(L)$ and $h$ is a smooth Hermitian metric on $L$ such that its Chern curvature form $R_h$ satisfies
$$R_h:=-\idd\log h=\omega.$$
Since $\frac{1}{2\pi}\omega$ lies in a positive integral class, the volume of $(X,\omega)$ is automatically non-collapsed.

For any $k\in\mathbb{N}_+$, using $\omega$ and $h$, we define an $L^2$ Hermitian inner product $\langle\cdot\ ,\cdot\rangle$ on the vector space $H^0(X,L^k)$ by setting
\begin{equation}
    \label{equation:Hermitian-inner-prod}
    \langle s_1,s_2 \rangle:=\int_X(s_1,s_2)_{h^k}\frac{(k\omega)^n}{n!},\ \forall s_1,s_2\in H^0(X,L^k).
\end{equation}


Using this Hermitian inner product, we can pick an orthonormal basis $s_0,s_1,...,s_{N_k}$ of $H^0(X,L^k)$, where $1+N_k=h^0(X,L^k)$. We define the Bergman kernel of $(X,\omega,L,h)$ with multiple $k$ to be
\begin{equation}
    \label{equation:definition-Bergman-Kernel}
    \rho_{\omega,k}(x):=\sum_{i=0}^{N_k}|s_i|^2_{h^k}(x),\ x\in X.
\end{equation}
Note that $\rho_{\omega,k}$ is independent of the choices of $h$ and the orthonormal basis.

It is well-known that we have the following asymptotic behavior of $\rho_{\omega,k}$ as $k\rightarrow\infty$:
\begin{equation}
    \label{equation:Bergman-Kernel-asymp}
    \rho_{\omega,k}\sim\frac{1}{(2\pi)^n}\big(1+\frac{S(\omega)}{2}k^{-1}+O(k^{-2})\big),
\end{equation}
where $S(\omega)$ denotes the scalar curvature of $\omega$ (see, e.g., \cite{T,R98,Z98,L00,C97}). 
In other words, the Bergman kernel is almost a constant
on a given polarized \K manifold once $k$ is sufficiently large. But a priori it is not clear at all how large this $k$ should be. In many circumstances, we are faced with a \textit{family} of polarized \K manifolds and we wish to derive a \emph{uniform bound} for their Bergman kernels with respect to some \emph{specific multiple} $k$. Such a uniform estimate for Bergman kernels is often referred to as the partial $C^0$ estimate, which first appeared in Tian's work \cite{T90} as a crucial tool in the study of the K\"ahler-Einstein problem.

The general question about the partial $C^0$ estimate is the following: given a family of polarized \K manifolds satisfying certain natural geometric conditions, to what extent can we control their Bergman kernels uniformly? For instance in \cite{T90'} Tian conjectured that for Fano manifolds with definite positive Ricci lower bound the Bergman kernel should have a uniform positive lower bound with respect to some bounded multiple $k$. It has been shown by Tian-Zhang \cite{TZ16} and Jiang \cite{J16} that, Tian's partial $C^0$ conjecture would hold if one can prove the Hamilton-Tian conjecture (which has been recently solved by Bamler \cite{B16} and Chen-Wang \cite{CW14} independently).

Regarding the partial $C^0$ estimate, significant progress and applications have been made by various authors in recent years.
Here we only mention one recent result proved by Liu--Sz\'ekelyhidi, which will be needed for our later discussions.

\begin{theorem}
[\cite{LS18}]
\label{thm:LS}
Given $n$ and $D<\infty$, there exists $k=k(n,D)$ and $b=b(n,D)$ such that the following holds. For any polarized \K manifold $(X,\omega,L,h)$ satisfying $\Ric(\omega)>-\omega$ and $\text{Diam}(X,\omega)\leq D$, one has
$$\rho_{\omega,k}\geq b>0.$$
\end{theorem}

This result extends the previous work of
Tian \cite{T90,T13} (in the K\"ahler--Einstein setting) and Donaldson-Sun \cite{DS14} (for general polarized pair with bounded Ricci curvature).
In this paper we will give some refinements and applications of Liu-Sz\'ekelyhidi's result. 

\subsection{Main results}
\hfill\\
Our first main result is a coarse asymptotic estimate for the Bergman kernel.

\begin{theorem}
\label{theorem:main-result}
Given $A<\infty$, there exists a large integer $D=D(n,A)$ and two constants $b=b(n,A)>0$, $B=B(n,A)<\infty$ such that the following holds.
Let $(X,\omega,L,h)$ be a polarized \K manifold that satisfies
\begin{enumerate}
     \item $C_S(X,\omega)\leq A$;
     \item $\Ric(\omega)>-\omega$.
\end{enumerate}
Then for any $k\in\mathbb{N}_+$, we have $b\leq\rho_{\omega,Dk}\leq B.$
\end{theorem}

\begin{remark}
In the above theorem, the condition $C_S(X,\omega)\leq A$ means that the following Sobolev inequality holds:
$$\bigg(\int_X |u|^{\frac{2n}{n-1}}\omega^n\bigg)^{\frac{n-1}{n}}\leq A\bigg(\int_Xu^2\omega^n+\int_X|\nabla u|^2\omega^n\bigg),\ \forall u\in W^{1,2}(X).$$

\end{remark}

Regarding Theorem \ref{theorem:main-result}, we note that, in \cite[Conjecture 5.15]{DS14}, a much sharper form is conjectured by Donaldson-Sun. For instance, the lower bound $b$ of $\rho_{\omega,Dk}$ should be almost optimal in the sense that, by increasing $D$ if necessary, we can choose $b$ to be arbitrarily close to the leading term $\frac{1}{(2\pi)^n}$ of \eqref{equation:Bergman-Kernel-asymp}. So Theorem \ref{theorem:main-result} gives a partial answer to this conjecture.

Our argument is inspired by the work of Bamler \cite{B16} and Chen-Wang \cite{CW14}. The key point is that, the conditions on the Sobolev constant and Ricci curvature are \textbf{preserved} if we rescale $(X,\omega,L,h)$ by large integers.
To be more precise, given any $k\in\mathbb{N}_+$, we may rescale $(X,\omega,L,h)$ in the following fashion. We put
$$\Tilde{\omega}:=k\omega,\ \Tilde{L}:=L^k,\ \Tilde{h}:=h^k.$$
Then $(X,\tilde{\omega},\tilde{L},\tilde{h})$ is again a polarized \K manifold (in this paper, any rescaling will be of this form). If $\omega$ satisfies
$$C_S(X,\omega)\leq A,\ 
\Ric(\omega)>-\omega.$$
Then it is direct to see that
$$C_S(X,\tilde{\omega})\leq A,\ \Ric(\tilde{\omega})>-\frac{1}{k}\tilde{\omega}.$$
In other words, if we define
$$
\mathcal{M}(n,A)=\left\{(X,\omega,L,h)\ \Bigg|\ %
\aligned
&\operatorname{dim}X=n\\
&\operatorname{Ric}(\omega)>-\omega\\
&C_s(X,\omega)\leq A 
\endaligned\Bigg.\right\},
$$
then the space $\mathcal{M}(n,A)$ is invariant under such rescalings.
It should be emphasized that, in this paper, whenever we raise $L$ to $L^k$, the underlying \K form $\omega$ will be rescaled to $k\omega$ accordingly, so that everything will work consistently (this justifies the volume form $(k\omega)^n$ in \eqref{equation:Hermitian-inner-prod}). 
This kind of treatment also appeared in, e.g., \cite{DS14,DS17}.

Another interesting rescaling property we shall use comes from the Bergman kernel itself. Indeed, $\rho_{\omega,k}$ enjoys the following rescaling property:
\begin{equation}
    \label{equation:rho-under-scaling}
    \rho_{l\omega,k}=\rho_{\omega,kl},\ \forall k,l\in\mathbb{N}_+.
\end{equation}
Here $\rho_{l\omega,k}$ denotes the Bergman kernel of $(X,l\omega,L^{l},h^l)$ with multiple $k$, so it follows directly from the definition that it coincides with the Bergman kernel of $(X,\omega,L,h)$ with multiple $kl$.
Note that, this simple property will play an important role in the proof of Theorem \ref{theorem:main-result}.

\begin{remark}
\label{remark:k^n}
If one prefers to use the following $L^2$ inner product:
$$ \langle s_1,s_2 \rangle=\int_X(s_1,s_2)_{h^k}\frac{\omega^n}{n!},\ \forall s_1,s_2\in H^0(X,L^k).$$
Then one would get an additional factor $k^n$ for the Bergman kernel $\rho_{\omega,k}$. In particular, in the statement of Theorem \ref{theorem:main-result}, one would have
$$bk^n\leq\rho_{\omega,Dk}\leq Bk^n,
 \forall k\in\mathbb{N}_+.$$
\end{remark}

The second main result of this paper is an application of Theorem \ref{thm:LS} to the partial $C^0$ estimate along the generalized K\"ahler-Ricci flow on Fano manifolds. Let $X$ be an $n$-dimensional compact K\"ahler manifold $X$ and $L$ an ample line bundle on $X$. Assume that there is a non-negative closed $(1,1)$ form $\alpha\in2\pi(c_1(X)-c_1(L))$. Note that in this setting, $X$ is automatically Fano.
We fix an initial metric $\omega_0\in 2\pi c_1(L)$ and consider the following generalized K\"ahler-Ricci flow
\begin{equation}
   \label{equation:generalized KRF}
    \frac{\partial}{\partial t}\omega_t=-\Ric(\omega_t)+\omega_t+\alpha
\end{equation}
starting from $\omega_0$. 
We show that the partial $C^0$ estimate holds along this flow.

\begin{theorem}
\label{theorem:main-tesult}
There exists a positive constant $b=b(n,\omega_0,\alpha)$ and a large integer $k=k(n,\omega_0,\alpha)$ such that
$$\rho_{\omega_t,k}\geq b>0$$
along the flow \eqref{equation:generalized KRF}.
\end{theorem}

To derive such estimates, in principle one needs to understand the space of generalized KRF, which is a highly nontrivial problem. Indeed, for the classical KRF (when $L=-K_X$ and $\alpha=0$), Theorem \ref{theorem:main-tesult} had remained unknown for some time and was studied by various authors (see, e.g., \cite{CW12,TZ16,J16}). As shown in \cite{TZ16}, the partial $C^0$ estimate along the KRF will follow from the Hamilton-Tian conjecture. For instance, relying on the $L^4$ bound of the Ricci curvature, Tian-Zhang \cite{TZ16} proved the Hamilton-Tian conjecture for dimension $n\leq3$ and hence the partial $C^0$ estimate follows in this case (for $n=2$, this was also proved by Chen-Wang \cite{CW12} using the $L^2$ bound of the curvature tensor). For general dimensions, the Hamilton-Tian conjecture was solved in \cite{B16,CW14} so the partial $C^0$ estimate holds as a consequence. So to prove Theorem \ref{theorem:main-tesult}, one possible approach is to establish the Hamilton-Tian compactness along the generalized flow \eqref{equation:generalized KRF}, which however seems to be out of reach at present. When studying the (generalized) KRF, the main difficulty comes from the lack of suitable curvature control. For instance the Ricci curvature along the flow usually does not have a uniform lower bound. So Theorem \ref{thm:LS} is not directly appliable in this case. To avoid this issue, we use the following strategy. By applying the Calabi-Yau theorem along the KRF we can get a family of K\"ahler metrics with positive Ricci curvature. Then we derive a uniform diameter upper bound for these metrics so that we can apply Theorem \ref{thm:LS} to get the partial $C^0$ estimate  (see Section \ref{sec:KRF} for more details).

As an application of Theorem \ref{theorem:main-tesult}, we obtain the following result, using Jiang's approach \cite{J16}.
\begin{theorem}\label{theorem:PC0-for-scalar-lower-bound}
Given $(X,\omega_0,\alpha)$ as above. Suppose that
\begin{itemize}
    \item The scalar curvature $R(\omega_0)$ satisfies $R(\omega_0)-tr_{\omega_0}\alpha\geq-\Lambda$ for some constant $\Lambda\geq0$;
    \item$(X,\omega_0)$ satisfies the following $L^2$-Sobolev inequality: 
    \begin{equation*}
        \bigg(\int_X|f|^{\frac{2n}{n-1}}\omega_0^n\bigg)^{\frac{n-1}{n}}\leq C_S \int_X|\nabla f|^2\omega^n_0
    \end{equation*}
    for any $f\in W^{1,2}(X)$ with $\int_Xf\omega^n_0=0$;
    \item Let $V=\int_X\omega^n_0$ be the volume of the K\"ahler class.
\end{itemize}
Then there exists a large integer $k$ and a constant $b>0$ such that
$$\rho_{\omega_0,k}\geq b>0.$$
Here $k$ and $b$ only depend on $n,\Lambda,C_S$ and $V$.
\end{theorem}

\begin{remark}
Many arguments contained in this paper should be well-known to experts. But we did not find the precise reference so we write down the details for reader's convenience.  Also note that, after completing the first draft, the author was informed that our approach to the partial $C^0$ estimate along the KRF had been aware of by Tian and Zhenlei Zhang as well.
\end{remark}

The rest of this paper is organized as follows.
In Section \ref{sec:asymp} we prove Theorem \ref{theorem:main-result} and extend the argument to K\"ahler-Ricci solitons. In Section \ref{sec:KRF} we prove Theorem \ref{theorem:main-tesult} and Theorem \ref{theorem:PC0-for-scalar-lower-bound}.

\smallskip
\textbf{Acknowledgments.}
The author wants to express his deep gratitude to Profs. Gang Tian, Xiaohua Zhu and Yanir Rubinstein for many inspiring discussions and for encouraging him to publish this work. Thanks also go to Wenshuai Jiang, Feng Wang, Zhenlei Zhang, Yalong Shi and Wangjian Jian for valuable comments.
The author is sponsored by the China post-doctoral grant BX20190014.

\section{Weak asymptotics of Bergman kernels}
\label{sec:asymp}
This section is devoted to the proof of Theorem \ref{theorem:main-result}. A similar result for K\"ahler-Ricci solitons will be proved in the end of this section as well. Let us begin with some preliminaries.
\subsection{Preparations}
\label{section:priliminaries}
\hfill\\
In this part, we collect some standard results in the literature that will be used in our argument.

First we recall that the Sobolev bound implies volume non-collapsing.
\begin{lemma}[\cite{H00}]
\label{lemma:kappa-noncollapsing}
For any $A<\infty$, there exists a constant $\kappa=\kappa(n,A)>0$ such that the following holds.
Let $(X,g)$ be a compact Riemannian manifold. Assume that $C_S(X,g)\leq A$, then we have
$$\vol(B(x,1))\geq \kappa,\ \forall x\in X.$$
\end{lemma}

Also recall that the Sobolev bound gives an upper bound for the norms of holomorphic sections.
\begin{lemma}
\label{lemma:bounds-on-sections}
For any $A<\infty$, there exists a constant $B=B(n,A)>0$ such that the following holds.
Let $(X,\omega,L,h)$ be an $n$-dimensional polarized \K manifold such that
$C_S(X,\omega)\leq A$.
Then for any section $s\in H^0(X,L)$, we have
$$||s||_{L^\infty}^2\leq B||s||^2_{L^2}.$$
\end{lemma}

\begin{proof}
We have
$$\Delta|s|^2_h=|\nabla s|^2_h-n|s|^2_h.$$
So the result follows from the standard Moser iteration.
\end{proof}

We also recall the following standard estimate (cf. \cite{DS14}).
\begin{lemma}
\label{lemma:rho_kl>rho_k}
For any $A<\infty$, we have the following fact. Let $(X,\omega,L,h)$ be an $n$-dimensional polarized \K manifold such that
$C_S(X,\omega)\leq A$. Then for any $k,l\in\mathbb{N}_+$, we have
$$\rho_{\omega,kl}\geq\frac{(\rho_{\omega,k})^l}{l^nB^{l-1}},$$
where $B$ is the constant in Lemma \ref{lemma:bounds-on-sections}.
\end{lemma}
\begin{proof}
Consider an arbitrary point $x\in X$. We may assume that $\rho_{\omega,k}(x)\neq0$, then there exists $s\in H^0(X,L^k)$ with
$$||s||^2_{L^2}=\int_X|s|^2_{h^k}\frac{(k\omega)^n}{n!}=1$$
such that
$$\rho_{\omega,k}(x)=|s|^2_{h^k}(x).$$
Now we look at $s^l\in H^0(X,L^{kl})$. Using Lemma \ref{lemma:bounds-on-sections}, it is clear that
$$||s^l||^2_{L^2}=\int_X(|s|_{h^k}^2)^l\frac{(kl\omega)^n}{n!}\leq l^n(||s||_{L^\infty}^2)^{l-1}||s||^2_{L^2}\leq l^nB^{l-1}.$$
So we get
$$\rho_{\omega,kl}(x)\geq\frac{|s^l|_{h^{kl}}^2(x)}{||s^l||^2_{L^2}}\geq\frac{(\rho_{\omega,k}(x))^l}{l^nB^{l-1}}.$$
\end{proof}

\subsection{Proof of Theorem \ref{theorem:main-result}}
\label{section:pf-of-main-result}
\hfill
\hfill

To prove Theorem \ref{theorem:main-result}, we will make use of the following local partial $C^0$ estimate obtained in \cite{LS18}.

\begin{proposition}(\cite[Proposition 3.1]{LS18})
\label{proposition:Gabor-Liu}Given $A<\infty$, there is a large integer $K_0=K_0(n,A)$ and two constants $\epsilon=\epsilon(n,A)>0$, $c=c(n,A)>0$ with the following property. Let $(X,\omega,L,h)$ be an $n$-dimensional polarized \K manifold such that
\begin{enumerate}
     \item $C_S(X,\omega)\leq A$;
     \item $\Ric(\omega)>-\epsilon\omega$.
\end{enumerate}
Suppose that $d_{GH}(B(p,\epsilon^{-1}),B(o,\epsilon^{-1}))<\epsilon$
for a metric cone $(V,o)$. Then there exists an integer $m\leq K_0$, such that
$$\rho_{\omega,m}(p)\geq c>0.$$
\end{proposition}

 The proof of this relies on a new $\epsilon$-regularity result \cite[Theorem 2.1]{LS18}, the \textit{peak section method} initiated in \cite{T} and the techniques in \cite{DS14,T13}. The basic philosophy is that, whenever the manifold is close to a metric cone, we can construct suitable holomorphic sections. Note that in \cite[Proposition 3.1]{LS18}, $(X,\omega)$ is assumed to be locally non-collapsed. This is guaranteed in our setting by Lemma \ref{lemma:kappa-noncollapsing}.

Now notice that, Proposition \ref{proposition:Gabor-Liu} gives the following compactness result.

\begin{proposition}
\label{proposition:exists-l<k-such-that-rho>b}
Given $A<\infty$, there is a large integer $K_1=K_1(n,A)$ and a constant $\eta=\eta(n,A)>0$ with the following property. Let $(X,\omega,L,h)$ be an $n$-dimensional polarized \K manifold such that
\begin{enumerate}
     \item $C_S(X,\omega)\leq A$;
     \item $\Ric(\omega)>-\omega$.
\end{enumerate}
Then for any point $p\in X$, there exists an integer $m\leq K_1$ such that
$$\rho_{\omega,m}(p)\geq\eta>0.$$
\end{proposition}

\begin{proof}
We argue by contradiction. Suppose that the statement is wrong, then there exists $A<\infty$ such that, for any $K_i\rightarrow\infty$ and $\eta_i\rightarrow0$, there exists a polarized sequence $(X_i,\omega_i,L_i,h_i)$ satisfying
\begin{enumerate}
     \item $C_S(X_i,\omega_i)\leq A$,
     \item $\Ric(\omega_i)>-\omega_i$,
\end{enumerate}
and there exists $p_i\in X_i$ such that for any $m\leq K_i$, we have
\begin{equation}
    \label{equation:contrad-assumption}
    \rho_{\omega_i,m}(p_i)\leq \eta_i.
\end{equation}

Note that by Lemma \ref{lemma:kappa-noncollapsing}, the sequence $(X_i,\omega_i,p_i)$ satisfies local non-collapsing condition. So the standard Cheeger-Colding theory works perfectly.
After passing to a subsequence, we may assume
$$(X_i,\omega_i,p_i)\xrightarrow{\text{pointed } GH}(Z,d,p_\infty).$$
Note that $Z$ does not have to be a metric cone. But the blow-up will always be. So we take a sequence of integers $l_j\rightarrow\infty$ and by passing to a subsequence we can assume that
$$(Z,\sqrt{l_j}d,p_\infty)\xrightarrow{\text{pointed } GH} (V,o)$$
for some metric cone $(V,o)$.
Now we take $\epsilon=\epsilon(n,A)$ from Proposition \ref{proposition:Gabor-Liu}. Then for any $j$ sufficiently large, we have
$$d_{GH}\bigg(B_{d_j}(p_\infty,\epsilon^{-1}),B(o,\epsilon^{-1})\bigg)<\frac{\epsilon}{2}.$$
Here $B_{d_j}(p_\infty,\epsilon^{-1})$ denotes the ball centered at $p_\infty$ measured with respect to the rescaled metric $d_j=\sqrt{l_j}d$. We now \emph{fix} such an $j$. Then we clearly have
$$(X_i,l_j\omega_i,p_i)\xrightarrow{\text{pointed } GH}(Z,d_j,p_\infty).$$
Thus for any $i$ large enough, we have
$$d_{GH}\bigg(B_{l_j\omega_i}(p_i,\epsilon^{-1}),B_{d_j}(p_\infty,\epsilon^{-1})\bigg)<\frac{\epsilon}{2}.$$
Here $B_{l\omega}(p_i,\epsilon^{-1})$ denotes the ball centered at $p_i$ measured with respect to the rescaled \K form $l_j\omega_i$.
Thus we see that
$$d_{GH}\bigg(B_{l_j\omega_i}(p_i,\epsilon^{-1}),B(o,\epsilon^{-1})\bigg)<\epsilon$$
for any sufficiently large $i$. By increasing $j$ if necessary, we may further assume that $1/l_j<\epsilon$. Then Proposition \ref{proposition:Gabor-Liu} can be applied to the polarized manifold $(X_i,l_j\omega_i,L_i^{l_j},h_i^{l_j})$ for sufficiently large $i$. So we can find $m_i\leq K_0=K_0(n,A)$ such that
$$\rho_{l_j\omega_i,m_i}(p_i)\geq c=c(n,A)>0,$$
with $K_0$ and $c$ determined by Proposition \ref{proposition:Gabor-Liu}. Now thanks to the rescaling property \eqref{equation:rho-under-scaling}, we arrive at
$$\rho_{\omega_i,l_jm_i}(p_i)\geq c>0,$$
contradicting our assumption \eqref{equation:contrad-assumption} whenever $i$ is large enough.
\end{proof}

Now we can apply Lemma \ref{lemma:rho_kl>rho_k} to refine the statement of Proposition \ref{proposition:exists-l<k-such-that-rho>b}.

\begin{proposition}
\label{proposition:exists-large-D}
Given $A<\infty$, there is a large integer $D=D(n,A)$ and a constant $b=b(n,A)>0$ with the following property. Let $(X,\omega,L,h)$ be an $n$-dimensional polarized \K manifold such that
\begin{enumerate}
     \item $C_S(X,\omega)\leq A$;
     \item $\Ric(\omega)>-\omega$.
\end{enumerate}
Then we have
$$\rho_{\omega,D}(p)\geq b>0,\  \forall p\in X.$$
\end{proposition}
\begin{proof}
We choose $D=(K_1)!$, where $K_1=K_1(n,A)$ is the integer determined in the previous proposition. So for any $m\leq K_1$, $D$ is divisible by $m$. Now for any $p\in X$, Proposition \ref{proposition:exists-l<k-such-that-rho>b} guarantees that there exists $m_p\leq K_1$ and $\eta=\eta(n,A)>0$ such that
$$\rho_{\omega,m_p}(p)\geq \eta>0.$$
Now applying Lemma \ref{lemma:rho_kl>rho_k}, we get
$$\rho_{\omega,D}(p)\geq \frac{(\rho_{\omega,m_p}(p))^{D/m_p}}{(D/m_p)^nB^{D/m_p-1}}\geq\frac{\min\{1,\eta^D\}}{D^nB^{D-1}}>0.$$
So we choose $b=\frac{\min\{1,\eta^D\}}{D^nB^{D-1}}$ and finish the proof.
\end{proof}

Finally, we are able to prove Theorem \ref{theorem:main-result}.
\begin{proof}[Proof of Theorem \ref{theorem:main-result}]
Let $(X,\omega,L,h)$ be a polarized \K manifold that satisfies
\begin{enumerate}
     \item $C_S(X,\omega)\leq A$;
     \item $\Ric(\omega)>-\omega$.
\end{enumerate}
For any $k\in\mathbb{N}_+$, if we put
$$\Tilde{\omega}:=k\omega,\ \Tilde{L}:=L^k,\ \Tilde{h}:=h^k.$$
Then we would get
$$C_S(X,\tilde{\omega})\leq A\ \text{and } \Ric(\tilde{\omega})>-\frac{1}{k}\tilde{\omega}.$$
So the upper bound of $\rho_{\omega,k}$ for each $k\in\mathbb{N}_+$ follows directly if we apply Lemma \ref{lemma:bounds-on-sections} to $(X,\tilde{\omega},\tilde{L},\tilde{h})$. For the lower bound, note that Proposition \ref{proposition:exists-large-D} can be applied to the polarized pair $(X,\tilde{\omega},\tilde{L},\tilde{h})$. So we find $D=D(n,A)$ and $b=b(n,A)>0$ such that
$$\rho_{\tilde{\omega},D}(p)\geq b>0,\ \forall p\in X.$$
Finally, the rescaling property \eqref{equation:rho-under-scaling} gives
$$\rho_{\omega,Dk}\geq b>0,$$
as desired.
\end{proof}

\subsection{Partial $C^0$ estimate for K\"ahler-Ricci solitons}
\label{section:application}
\hfill


We can also extend the argument in the previous subsection to K\"ahler-Ricci solitons.
More specifically, we have the following result, refining \cite[Theorem 1.1]{PSS15}.
\begin{theorem}
There exists $D=D(n)<\infty$, $b=b(n)>0$ and $B=B(n)$ with the following property. Let $X$ be an $n$-dimensional Fano manifold and suppose that 
$\omega\in2\pi c_1(X)$ satisfies
$$\Ric(\omega)=\omega+\mathcal{L}_\xi\omega$$
for some holomorphic vector field $\xi$ on $X$. Namely $(X,\omega,\xi)$ is a K\"ahler-Ricci soliton. Then we have
$$b\leq\rho_{\omega,Dk}\leq B,\ \forall k\in\mathbb{N}_+.$$
Here $\rho_{\omega,Dk}$ denotes the Bergman kernel of $(X,-K_X,\omega)$ with multiple $Dk$.
\end{theorem}

The techniques are more or less standard (following \cite{Z10,TZ12,PSS15}).
We outline a proof for reader's convenience.
\begin{proof}
Given a K\"ahler-Ricci soliton $(X,\omega,\xi)$, we can find a potential function $u\in C^\infty(X,\mathbb{R})$ such that
$$\Ric(\omega)=\omega-\idd u,\ \text{with}\ u_{ij}=u_{\bar{i}\bar{j}}=0.$$
As pointed out in \cite{PSS18}, we can assume that
$$|u|+|\nabla u|^2+|\Delta u|\leq C_1\ \text{for some }C_1=C_1(n).$$
And also, we have
$$C_S(X,\omega)\leq A\ \text{for some }A=A(n),$$
which takes care of the upper bound for $\rho_{\omega,k}$ by Lemma \ref{lemma:bounds-on-sections}. Now we derive the lower bound.
We recall Zhenlei Zhang's trick (see \cite{Z10}). Put
$$\eta:=e^{-\frac{u}{n-1}}\omega,$$
then we have
$$-C_2\leq\Ric(\eta)\leq C_2,\ \text{for some }C_2=C_2(n).$$
Note that, these estimates are the key ingredients to the partial $C^0$ estimate in \cite[Theorem 1.1]{PSS15}.
Now the important observation is that, these estimates are \textbf{preserved} if we rescale $\omega$ by some integers greater than 1. 

To be more precise, for any $k\in\mathbb{N}_+$, if we put
$$\tilde{\omega}:=k\omega,$$
then we would get
$$|u|+k|\nabla_{\tilde{\omega}}u|^2+k|\Delta_{\tilde{\omega}}u|\leq C_1\ \text{and}\ C_S(X,\tilde{\omega})\leq A.$$
Meanwhile, if we put
$$\tilde{\eta}:=e^{-\frac{u}{n-1}}\tilde{\omega},$$
then we have
$$-\frac{C_2}{k}\leq\Ric(\tilde{\eta})\leq \frac{C_2}{k}.$$
So as one can see, rescaling makes things better.

Now the proof can be carried out in the same manner as we did in Section \ref{section:pf-of-main-result}. The key result is Proposition \ref{proposition:PSS-local-PC0} below (compare Proposition \ref{proposition:Gabor-Liu}). With this in hand, we can then follow the argument of Proposition \ref{proposition:exists-l<k-such-that-rho>b} (using the Cheeger-Colding-Tian theory developed in \cite{Z10,TZ12})
to obtain a large integer $K=K(n)$ and $\eta=
\eta(n)>0$ such that, for any rescaled K\"ahler-Ricci soliton $(X,\tilde{\omega},\xi)$ and any point $p\in X$, there exists $m_p\leq K$ such that
$$\rho_{\tilde{\omega},m_p}(p)\geq \eta>0.$$
Then the same argument as in the proof of Proposition \ref{proposition:exists-large-D} gives a large integer $D=D(n)$ and $b=b(n)>0$ such that
$$\rho_{\tilde{\omega},D}(p)\geq b>0,\ \forall p\in X.$$
Finally the rescaling property \eqref{equation:rho-under-scaling} completes the proof.
\end{proof}

\begin{proposition}
\label{proposition:PSS-local-PC0}
There is a large integer $K_0=K_0(n)$ and two constants $\epsilon=\epsilon(n)>0$, $c=c(n)>0$ with the following property. Let $(X,\tilde{\omega})$ be an $n$-dimensional Fano manifold with $\tilde{\omega}\in 2\pi k c_1(X)$ for some $k\in \mathbb{N}_+$ such that $1/k<\epsilon$. Assume that there exists a potential function $u\in C^\infty(X,\mathbb{R})$ such that
$$\Ric(\tilde{\omega})=\frac{1}{k}\tilde{\omega}-\idd u,\ \text{with}\ u_{ij}=u_{\bar{i}\bar{j}}=0.$$
Namely $\tilde{\omega}$ is a rescaled K\"ahler-Ricci soliton metric (with sufficiently large scaling factor $k$).
Also assume that 
$$d_{GH}\big(B_{\tilde{\omega}}(p,\epsilon^{-1}),B(o,\epsilon^{-1})\big)<\epsilon$$
for a metric cone $(V,o)$. Then there exists an integer $m\leq K_0$, such that
$$\rho_{\tilde{\omega},m}(p)\geq c>0.$$
\end{proposition}

\begin{proof}
This is essentially contained in \cite[Section 5]{PSS15}. We argue by contradiction. Suppose that for $k_i\rightarrow\infty$, we have a sequence of blowing-up K\"ahler-Ricci solitons $(X_i,\tilde{\omega_i},\xi_i)$ with
$$\Ric(\tilde{\omega_i})=\frac{1}{k_i}\tilde{\omega_i}-\idd u_i$$
and
$$(X,\tilde{\omega_i},p_i)\xrightarrow{pointed\ GH}(V,o)$$
for some metric cone $(V,o)$. Then by \cite{TZ12}, we know that $V$ is Ricci flat away from a closed singular set with codimension at least 4 and the convergence takes place in $C^\infty$ topology on the regular part. Then the argument in \cite{DS14,T13} can be applied in this setting (see also \cite{TZ16}) to deduce that, there exists $K<\infty$ and $c>0$ such that, for any sufficiently large $i$, there exists some $m_i\leq K$ such that $\rho_{\tilde{\omega_i},m_i}(p_i)\geq c>0.$
\end{proof}

\section{Partial $C^0$ estimate along the generalized K\"ahler Ricci flow}
\label{sec:KRF}

In this section, we will focus on the \textit{generalized} K\"ahler-Ricci flow (KRF) on Fano manifolds.
The purpose of this section is to prove Theorem \ref{theorem:main-tesult} and Theorem \ref{theorem:PC0-for-scalar-lower-bound}. 
\subsection{Preliminaries on the generalized KRF}
\label{section:preliminaries}
\hfill\\
In this part, we recall some standard results of the generalized KRF. These results were well-established for KRF and were later extended to the generalized setting in \cite{L13,CS16}.

Let us recall the setup.
Let $X$ be an $n$-dimensional compact K\"ahler manifold and $L$ an ample line bundle on $X$. Assume that there is a non-negative closed $(1,1)$ form $\alpha\in2\pi(c_1(X)-c_1(L))$. Note that in this setting, $X$ is automatically Fano.
We fix an initial metric $\omega_0\in 2\pi c_1(L)$ and consider the generalized KRF \eqref{equation:generalized KRF}:
\begin{equation*}
    \frac{\partial}{\partial t}\omega_t=-\Ric(\omega_t)+\omega_t+\alpha
\end{equation*}
starting from $\omega_0$.

This flow preserves the cohomology class of $\omega_t$ and exists for all time. Along the flow, we choose a family of smooth Hermitian metrics $h_t$ on the line bundle $L$ such that the curvature form of $h_t$ satisfies
\begin{equation}
    \label{equation:}
    -\idd\log h_t=\omega_t.
\end{equation}

For any K\"ahler form $\omega\in[\omega_0]$, let $f_{\omega}\in C^\infty(X,\mathbb{R})$ be the generalized Ricci potential of $\omega$, which is uniquely determined by
\begin{equation}
    \label{equation:generalized-Ricci-potential}
    \Ric(\omega)=\omega+\alpha+\idd f_{\omega},\ \int_Xe^{f_{\omega}}\omega^n=V,
\end{equation}
where $V=(2\pi c_1(L))^n$ is the volume of the K\"ahler class.

The following result is essentially due to Perelman.
\begin{theorem}[\cite{SesumTian,L13,CS16}]
\label{theorem:Perelman-estimate-along-generalized KRF}
Along the flow \eqref{equation:generalized KRF}, there exists a uniform constant $C$ such that
$$|f_{\omega_t}|+|\nabla f_{\omega_t}|+|\Delta f_{\omega_t}|+diam(X,\omega_t)\leq C.$$
Here the gradient, Laplacian and the norms are all taken with respect to the evolving metric $\omega_t$. The constant $C$ only depends on the dimension $n$, the volum $V$, the $L^2$-Sobolev constant of $(X,\omega_0)$, $|\nabla f_{\omega_0}|$ and $|\Delta f_{\omega_0}|$.
\end{theorem}

We also have a uniform Sobolev inequality along the generalized KRF.
\begin{theorem}[\cite{Z07,Y15,L13,CS16}]
\label{theorem:Sobolev-inequality}
Along the flow \eqref{equation:generalized KRF}, there exists a uniform constant $C_S$ such that for any $u\in W^{1,2}(X)$, we have
$$\bigg(\int_X |u|^{\frac{2n}{n-1}}\omega_t^n\bigg)^{\frac{n-1}{n}}\leq C_S\bigg(\int_Xu^2\omega_t^n+\int_X|\nabla u|^2\omega_t^n\bigg).$$
Here $C_S$ only depends on the dimension $n$, the volum $V$, the $L^2$-Sobolev constant of $(X,\omega_0)$, $|\nabla f_{\omega_0}|$ and $|\Delta f_{\omega_0}|$.
\end{theorem}
 
The following is Futaki's weighted Poincar\'e inequality in the generalized setting (which still holds since $\alpha$ is nonnegative). Note that this lemma will play a crucial role in our proof of Theorem \ref{theorem:main-tesult}.
\begin{lemma}[\cite{F98,TianZhu07,L13,CS16}]
\label{lemma:Poincare-inequality}
Let $\omega\in[\omega_0]$ be any K\"ahler form. Then for any function $u\in W^{1,2}(X)$ with $\int_X u e^{f_\omega}\omega^n=0$, we have
$$\int_X u^2e^{f_\omega}\omega^n\leq\int_X|\nabla u|^2e^{f_\omega}\omega^n.$$
\end{lemma}

\subsection{Applying Ricci inverse operator along the flow}
\label{section:calabi-yau-along-flow}
\hfill\\
The flow \eqref{equation:generalized KRF} preserves the cohomological class of $\omega_t$, so we have
$$\omega_t+\alpha\in 2\pi c_1(X).$$
Therefore we can apply the Calabi-Yau theorem to obtain a family of K\"ahler forms $\eta_t\in[\omega_0]$ such that
\begin{equation}
    \label{equation:Calabi-Yau}
    Ric(\eta_t)=\omega_t+\alpha.
\end{equation}
(Similar consideration also appeared in Rubinstein's work \cite[Section 9]{R08}.)
Note that $\eta_t$ satisfies the following Monge-Amp\`ere equation:
\begin{equation}
    \label{equation:MA-equation}
    \eta_t^n=e^{f_{\omega_t}}\omega_t^n.
\end{equation}
Since $\eta_t$ and $\omega_t$ are in the same K\"ahler class, we can write
\begin{equation}
    \label{equation:eta=omega-idd-phi}
    \eta_t=\omega_t-\idd\phi_t
\end{equation}
for some $\phi_t\in C^\infty(X,\mathbb{R})$. It is clear that $\phi_t$ and $f_{\eta_t}$ only differ by a constant (recall \eqref{equation:generalized-Ricci-potential}).
The main result of this section is the following

\begin{proposition}
\label{theorem:osc-diam-bound}
There exists a uniform constant $C=C(n,\omega_0,\alpha)$ depending only on the initial data such that
$$osc_X\phi_t+diam(X,\eta_t)\leq C.$$
\end{proposition}

The proof will be divided into two parts. We first prove the oscillation estimate following Yau's approach and then derive the diameter bound.

\begin{lemma}
\label{lemma:Yau-osc-estimate}
There exists a uniform constant $C=C(n,\omega_0,\alpha)$ depending only on the initial data such that
$$osc_X\phi_t\leq C.$$
\end{lemma}

\begin{proof}
This is standard. We include a proof for reader's convenience, following the exposition in \cite{SongWeinkov}.
For simplicity, we will abbreviate the subscript $t$. We may assume that
$$\int_X\phi \eta^n=\int_X\phi e^{f_\omega}\omega^n=0.$$
So it is enough to derive a uniform bound for $||\phi||_{C^0}$.

First, we have
\begin{equation*}
    \begin{aligned}
        \int_X|\phi|\omega^n\geq\int_X-\phi\omega^n&=\int_X\phi(\eta^n-\omega^n)\\
        &=-\int_X\phi\idd\phi\wedge\sum_{i=0}^{n-1}\eta^i\wedge\omega^{n-1-i}\\
        &=\int_X\sqrt{-1}\partial\phi\wedge\bar{\partial}\phi\wedge\sum_{i=0}^{n-1}\eta^i\wedge\omega^{n-1-i}\\
        &\geq\int_X\sqrt{-1}\partial\phi\wedge\bar{\partial}\phi\wedge\omega^{n-1}\\
        &=\frac{1}{n}\int_X|\nabla\phi|^2\omega^n.\\
    \end{aligned}
\end{equation*}
Then we apply Futaki's weighted Poincar\'e inequality (Lemma \ref{lemma:Poincare-inequality}) and H\"older inequality to derive
\begin{equation*}
\begin{aligned}
     \int_X\phi^2e^{f_\omega}\omega^n&\leq\int_X|\nabla \phi|^2e^{f_\omega}\omega^n\\
     &\leq e^{\sup_X f_\omega}\int_X|\nabla\phi|^2\omega^n
     \leq ne^{\sup_X f_\omega}\int_X|\phi|\omega^n\\
     &\leq ne^{\osc_Xf_\omega}\int_X|\phi|e^{f_\omega}\omega^n
     \leq ne^{\osc_Xf_\omega}\bigg(\int_X|\phi|^2e^{f_\omega}\omega^n\bigg)^{\frac{1}{2}}\bigg(\int_Xe^{f_\omega}\omega^n\bigg)^{\frac{1}{2}}
\end{aligned}
\end{equation*}
Thus we get
$$\int_X\phi^2e^{f_\omega}\omega^n\leq n^2e^{2\osc_Xf_\omega}V.$$
Now using the fact that $|f_\omega|$ is uniformly bounded (recall Theorem \ref{theorem:Perelman-estimate-along-generalized KRF}), we immediately get an $L^2$ bound:
\begin{equation}
    \label{equation:L^2-bound}
    ||\phi||_{L^2(\omega)}=\bigg(\int_X\phi^2\omega^n\bigg)^{\frac{1}{2}}\leq C_1
\end{equation}
for some constant $C_1=C(n,\omega_0,\alpha)$.

Now for any $p\geq 1$, using the fact the $x|x|^{p-1}$ is a differentiable function with derivative $p|x|^{p-1}$, we have
\begin{equation*}
    \begin{aligned}
        (e^{\sup_Xf_\omega}-1)\int_X|\phi|^p\omega^n&\geq\int_X|\phi|^p(\eta^n-\omega^n)\\
        &\geq\int_X\phi|\phi|^{p-1}(\eta^n-\omega^n)\\
        &=-\int_X\phi|\phi|^{p-1}\idd\phi\wedge\sum_{i=0}^{n-1}\eta^i\wedge\omega^{n-1-i}\\
        &=p\int_X|\phi|^{p-1}\sqrt{-1}\partial\phi\wedge\bar{\partial}\phi\wedge\sum_{i=0}^{n-1}\eta^i\wedge\omega^{n-1-i}\\
        &\geq\frac{4p}{(p+1)^2}\int_X\sqrt{-1}\partial(\phi|\phi|^{\frac{p-1}{2}})\wedge\bar{\partial}(\phi|\phi|^{\frac{p-1}{2}})\wedge\omega^{n-1}\\
        &=\frac{4p}{(p+1)^2n}\int_X|\nabla (\phi|\phi|^{\frac{p-1}{2}})|^2\omega^n.
    \end{aligned}
\end{equation*}
It then follows that (keeping in mind that $|f_\omega|$ is uniformly bounded)
\begin{equation*}
    \int_X|\nabla (\phi|\phi|^{\frac{p-1}{2}})|^2\omega^n\leq C_2p\int_X|\phi|^p\omega^n
\end{equation*}
for some constant $C_2=C(n,\omega_0,\alpha)$. Then applying Sobolev inequality (Theorem \ref{theorem:Sobolev-inequality}) to the function $\phi|\phi|^{\frac{p-1}{2}}$, we get
\begin{equation}
    \label{equation:Sobolev-for-phi^p}
    \bigg(\int_X|\phi|^{\frac{n(p+1)}{n-1}}\bigg)^{\frac{n-1}{n}}\leq C_S\bigg(C_2p\int_X|\phi|^p\omega^n+\int_X|\phi|^{p+1}\omega^n\bigg).
\end{equation}
Notice that
\begin{equation*}
    \begin{aligned}
        \int_X|\phi|^p\omega^n&=\int_{|\phi|<1}|\phi|^p\omega^n+\int_{|\phi|\geq1}|\phi|^p\omega^n\\
        &\leq\int_{|\phi|<1}\omega^n+\int_{|\phi|\geq1}|\phi|^{p+1}\omega^n
        \leq V+\int_X|\phi|^{p+1}\omega^n.
    \end{aligned}
\end{equation*}
So it follows from \eqref{equation:Sobolev-for-phi^p} that
$$ \bigg(\int_X|\phi|^{\frac{n(p+1)}{n-1}}\bigg)^{\frac{n-1}{n}}\leq C_3p\bigg(1+\int_X|\phi|^{p+1}\omega^n\bigg)\leq 2C_3p\max\bigg\{1,\int_X|\phi|^{p+1}\omega^n\bigg\}$$
for some constant $C_3=C(n,\omega_0,\alpha)$. Thus we obtain
\begin{equation*}
    \max\bigg\{1,||\phi||_{L^{\frac{n(p+1)}{n-1}}(\omega)}\bigg\}\leq (2C_3)^{\frac{1}{p+1}}p^{\frac{1}{p+1}}\max\bigg\{1,||\phi||_{L^{p+1}(\omega)}\bigg\}
\end{equation*}
Then standard Moser iteration gives
$$\max\bigg\{1,||\phi||_{L^\infty}\bigg\}\leq C_4\max\bigg\{1,||\phi||_{L^2(\omega)}\bigg\}$$
for some constant $C_4=C(n,\omega_0,\alpha)$. Combining this with \eqref{equation:L^2-bound}, we finish the proof.
\end{proof}

Since the generalized Ricci potential $f_{\eta_t}$ and $\phi_t$ only differ by a constant, the next result follows immediately from Lemma \ref{lemma:Yau-osc-estimate}.
\begin{corollary}
\label{corollary:bound-on-ricci-potential-for-eta}
There exists some constant $C=C(n,\omega_0,\alpha)>0$ such that
$$|f_{\eta_t}|\leq C$$
\end{corollary}

Now we are ready to prove the following diameter bound.
\begin{lemma}
\label{lemma:diameter-bound}
There exists some constant $C=C(n,\omega_0,\alpha)>0$ such that
$$diam(X,\eta_t)\leq C.$$
\end{lemma}

\begin{proof}
For simplicity, we abbreviate the subscript $t$.
By \eqref{equation:Calabi-Yau}, it is clear that
$$\Ric(\eta)>0.$$
If one can prove a uniform $C^2$ estimate for the Monge-Amp\`ere equation \eqref{equation:MA-equation}, then one would get $\Ric(\eta)\geq c>0$ and the diameter bound follows readily from Myer's theorem. But this approach does not seem to work in our setting since we do not have enough curvature control along the flow. So here we use a different strategy. 

We put
$$d:=diam(X,\eta),$$
and assume that $d=d_\eta(p,q)$ for two points $p,q\in X$, where $d_\eta$ is the distance function induced by $\eta$. We define
$$d_1(x):=d_\eta(x,p),\ d_2(x):=d_\eta(x,q),\ x\in X.$$
Triangle inequality simply gives
$$d_1(x)+d_2(x)\geq d,\ x\in X.$$
Integrating both sides against the volume form $e^{f_\eta}\eta^n$, we get
$$\frac{1}{V}\int_Xd_1e^{f_\eta}\eta^n+\frac{1}{V}\int_Xd_2e^{f_\eta}\eta^n\geq d.$$
So we may assume that
$$\overline{d_1}:=\frac{1}{V}\int_Xd_1e^{f_\eta}\eta^n\geq\frac{d}{2}.$$
Now applying Lemma \ref{lemma:Poincare-inequality} to $d_1(x)-\overline{d_1}$, we get
$$\int_X|d_1-\overline{d_1}|^2e^{f_\eta}\eta^n\leq\int_X|\nabla_\eta d_1|_\eta^2e^{f_\eta}\eta^n\leq C(n)\exp(\sup_Xf_\eta)\vol(B_\eta(p,d)).$$
On the other hand, since $|d_1-\overline{d_1}|\geq\frac{d}{4}$ on the ball $B_\eta(p,\frac{d}{4})$, we have
$$\int_X|d_1-\overline{d_1}|^2e^{f_\eta}\eta^n\geq \int_{B_\eta(p,\frac{d}{4})}|d_1-\overline{d_1}|^2e^{f_\eta}\eta^n\geq\frac{d^2}{C(n)}\exp(\inf_Xf_\eta)\vol(B_\eta(p,\frac{d}{4})).$$
Thus we get
$$d^2\leq C(n)\exp(\osc_X f_\eta)\frac{\vol(B_\eta(p,d))}{\vol(B_\eta(p,\frac{d}{4}))}.$$
Using $Ric(\eta)>0$ and relative volume comparison, we have
$$d^2\leq C(n)\exp(\osc_X f_\eta)4^{2n}.$$
Then the desired diameter bound follows from Corollary \ref{corollary:bound-on-ricci-potential-for-eta}.
\end{proof}



\subsection{Proof of Theorem \ref{theorem:main-tesult}}
\label{section:proof-of-main-result}
\hfill\\
As shown in the previous section, along the flow \eqref{equation:generalized KRF}, if we consider the K\"ahler form $\eta_t\in[\omega_0]$ such that $\Ric(\eta_t)=\omega_t+\alpha$, then
we have
$$Ric(\eta_t)>0,\ \vol(X,\eta_t)=C(n,[\omega_0]),\ diam(X,\eta_t)\leq C(\omega_0,\alpha).$$
So now we are in the setting where we can directly apply Theorem \ref{thm:LS}. We obtain the following
\begin{corollary}
\label{theorem:PC0-for-eta}
There exists a positive constant $b=b(n,\omega_0,\alpha)$ and a large integer $k=k(n,\omega_0,\alpha)$ such that
$$\rho_{\eta_t,k}\geq b>0$$
along the flow \eqref{equation:generalized KRF}.
\end{corollary}

To finish the proof of Theorem \ref{theorem:main-tesult}, it is enough to use the following simple fact.
\begin{lemma}
\label{lemma:equiv-of-Bergman-kernel}
For each $k\in\mathbb{N}$, there exists a constant $C=C(n,k,\omega_0,\alpha)>0$ such that
$$C^{-1}\rho_{\eta_t,k}\leq\rho_{\omega_t,k} \leq C\rho_{\eta_t,k}$$
along the flow \eqref{equation:generalized KRF}.
\end{lemma}
\begin{proof}
Given a family of smooth Hermitian metrics $\{h_t\}$ on $L$ with
$$-\idd\log h_t=\omega_t,$$
we consider
$$\widetilde{h_t}:=e^{f_{\eta_t}}h_t.$$
Then it is clear that
$$-\idd\log\widetilde{h_t}=\eta_t.$$
Now recall that we have the following uniform control (cf. Theorem \ref{theorem:Perelman-estimate-along-generalized KRF} and Corollary \ref{corollary:bound-on-ricci-potential-for-eta})
$$|f_{\omega_t}|+|f_{\eta_t}|\leq C(n,\omega_0,\alpha).$$
So the Hermitian metrics $h_t$ and $\widetilde{h_t}$ are fiber-wise comparable. Meanwhile, as volume forms, $\omega_t^n$ and $\eta^n_t$ are comparable as well (recall \eqref{equation:MA-equation}). So it follows easily from the definition of Bergman kernel that $\rho_{\omega_t,k}$ and $\rho_{\eta_t,k}$ are comparable as desired.
\end{proof}

\subsection{Proof of Theorem \ref{theorem:PC0-for-scalar-lower-bound}}
\label{section:Jiang}
\hfill\\
In this part we give an application of Theorem \ref{theorem:main-tesult}, following Jiang's work \cite{J16} closely (see also \cite{JWZ17}).
Our setup is as follows. Let $X$ be an $n$-dimensional compact K\"ahler manifold with an ample line bundle $L$. Let $\omega\in 2\pi c_1(L)$ be an K\"ahler form with $V=\int_X\omega^n$ being the volume.  Suppose that $(X,\omega)$ satisfies the following two conditions:
\begin{itemize}
    \item there exists a closed nonnegative $(1,1)$ form $\alpha\in2\pi c_1(-K_X-L)$ such that the scalar curvature $R(\omega)$ satisfies $R(\omega)-\tr_{\omega}\alpha\geq-\Lambda$ for some constant $\Lambda\geq0$;
    \item$(X,\omega)$ satisfies the following $L^2$-Sobolev inequality: 
    \begin{equation}
        \label{equation:strong-Sobolev-ineq}
        \bigg(\int_X|f|^{\frac{2n}{n-1}}\omega^n\bigg)^{\frac{n-1}{n}}\leq C_S\int_X|\nabla f|^2\omega^n
    \end{equation}
    for any $f\in W^{1,2}(X)$ with $\int_Xf\omega^n=0$.
\end{itemize}
Then we have the following partial $C^0$ estimate:
\begin{theorem}(Theorem \ref{theorem:PC0-for-scalar-lower-bound})
\label{theorem:Jiang}
One has
$$\rho_{\omega,k}\geq b>0$$
for some $k,b$ only depending on $n,V,\Lambda$ and $C_S$. 
\end{theorem}
The proof of this result is essentially contained in \cite{J16} and there are two main ingredients--- the regularization property of Ricci flows and the fact that the Bergman kernels at different time slices are comparable.
Note that similar argument was also exploited in \cite{JWZ17} by Jiang-Wang-Zhu. A simple observation is that, all the estimates in \cite{J16} hold analogously for the generalized KRF if we replace the scalar curvature $R$ by the twisted scalar curvature $R-\tr_\omega\alpha$ in the argument. And the proof of Theorem \ref{theorem:Jiang} is morally the same as the one for \cite[Theorem 1.5]{J16}. Note that in the statement of \cite[Theorem 1.5]{J16}, one can replace Ricci lower bound and diameter upper bound by other geometric conditions, since these two bounds are essentially used to get the lower bound of scalar curvature, the Sobolev inequality and the lower bound of Green's function. In our setting we use the inequality \eqref{equation:strong-Sobolev-ineq} to replace the bounds on Ricci and diameter and the argument in \cite{J16} works identically for our purpose. So in the following we only outline the proof, omitting some details.

We consider the generalized K\"ahler-Ricci flow
\begin{equation*}
    \frac{\partial}{\partial t}\omega_t=-\Ric(\omega_t)+\omega_t+\alpha
\end{equation*}
starting from $\omega$. Then the lower bound of $R(\omega)-\tr_\omega\alpha$ and the Sobolev inequality \eqref{equation:strong-Sobolev-ineq} will give us the following Sobolev inequality along the flow:
\begin{equation}
    \label{equation:Ye-Sobolev-along-flow}
    \bigg(\int_X|f|^{\frac{2n}{n-1}}\omega_t^n\bigg)^{\frac{n-1}{n}}\leq A\bigg(\int_X|\nabla f|^2\omega^n_t+(R(\omega_t)-\tr_{\omega_t}\alpha+B)\int_Xf^2\omega_t^2\bigg)
\end{equation}
for any $f\in W^{1.2}(X)$, where $A,B$ are positive constants only depending on $n,V,\Lambda$ and $C_S$ (cf. \cite{Y15,Z07,CS16}). Then we can follow the argument in \cite[Section 2,3]{J16} (see also \cite{JWZ17}) to deduce that
   $$ |\Delta f_{\omega_t}|+|\nabla f_{\omega_t}|^2\leq\frac{C}{t^{n+1}},\ t\in(0,1],$$
where $f_{\omega_t}$ is the generalized Ricci potential along the flow and the constant $C$ only depends on $n,V,\Lambda$ and $C_S$. Note that the lower bound of Green's function for $(X,\omega)$ is also involved in this estimate (see \cite[(3.6)]{J16}), which can be controlled in our setting by $n, V$ and $C_S$ (cf. Lemma \ref{lemma:green-function-lower-bound}). Then applying Theorem \ref{theorem:Perelman-estimate-along-generalized KRF} to the flow $\{\omega_t\}$ for $t\in[\frac{1}{2},1]$, we obtain
$$|f_{\omega_t}|+ |R(\omega_t-\tr_{\omega_t}\alpha)|^2\leq C,\ t\in[\frac{1}{2},1],$$
where $C$ only depends on $n,V,\Lambda$ and $C_S$. Now we can go through the proof of Theorem \ref{theorem:main-tesult} to find that 
$$\rho_{\omega_t,k}\geq b>0,\ t\in[\frac{1}{2},1]$$
for some $k,b$ only depending on $n,V,\Lambda$ and $C_S$. Finally, following the proof of \cite[Theorem 5.8]{J16}, we get
$$\rho_{\omega,k}\geq C^{-1}\rho_{\omega_t,k}\geq C^{-1}b>0,
 t\in[\frac{1}{2},1]$$
for some constant $C=C(n,V,\Lambda,C_S)>0$. So we finish the proof of Theorem \ref{theorem:Jiang}.

\begin{remark}
Form the proof of Theorem \ref{theorem:Jiang}, one can see that the non-negative form $\alpha$ itself only appears as an auxiliary term and does not play much roles in the argument. So it is likely that the condition on the lower bound of $R(\omega)-\tr_\omega\alpha$ can be replaced by other geometric conditions.
\end{remark}

\appendix
\section{Lower bound of Green's function}
\begin{lemma}[\cite{P12}]
\label{lemma:green-function-lower-bound}
Let $(X,g)$ be an $m$-dimensional Riemannian manifold ($m\geq3$) such that
the following $L^2$-Sobolev inequality holds: 
    \begin{equation}
        \label{equation:strong-Sobolev-ineq}
        \bigg(\int_X|f|^{\frac{2m}{m-2}}dV_g\bigg)^{\frac{m-2}{m}}\leq C_S\int_X|\nabla f|^2dV_g
    \end{equation}
    for any $f\in W^{1,2}(X)$ with $\int_XdV_g=0$.
Then the Green's function $G(x,y)$ of $(X,g)$ is bounded from below by a constant only depending on $m,\operatorname{Vol}(X,g)$ and $C_S$.
\end{lemma}
\begin{proof}
We sketch a proof for reader's convenience. Let
$$H(x,y,t):=\frac{1}{V}+\sum_{i=1}^\infty e^{-\lambda_it}\phi_i(x)\phi_i(y)$$
be the heat kernel of $(X,g)$. Here $\lambda_1\leq\lambda_2\leq...$ are the eigenvalues of the Laplacian $\Delta$ of $g$ and $\phi_i$'s are the corresponding eigenfunctions (i.e. $\Delta\phi_i=-\lambda_i\phi_i$) such that
$$\int_X\phi_i\phi_jdV_g=\delta_{ij}.$$
Then the Green's function $G(x,y)$ is given by
$$G(x,y):=\int_0^\infty(H(x,y,t)-\frac{1}{V})dt.$$
To get a lower bound of $G(x,y)$, it suffices to prove the following standard fact:
$$\bigg|H(x,y,t)-\frac{1}{V}\bigg|\leq\frac{C(n,C_S)}{t^{m/2}},\ t>0.$$

To this end, we put
$$H_1(x,y,t):=H(x,y,t)-\frac{1}{V}=\sum_{i=1}^\infty e^{-\lambda_it}\phi_i(x)\phi_i(y).$$
Then it is easy to verify
$$H_1(x,x,2t)=\int_XH_1(x,y,t)^2dV_g(y).$$
Taking time derivative, we get
\begin{equation}
\label{equation:dt-of-H=-2-int}
    \partial_tH_1(x,x,2t)=-2\int_X|\nabla_yH_1(x,y,t)|^2dV_g(y).
\end{equation}
On the other hand we have
\begin{equation*}
    \begin{aligned}
        H_1(x,x,2t)&=\int_XH_1^2(x,y,t)dV_g(y)\\
        &\leq\bigg(\int_X|H_1(x,y,t)|^{\frac{2m}{m-2}}dV_g(y)\bigg)^{\frac{m-2}{m+2}}\bigg(\int_X|H_1(x,y,t)|dV_g(y)\bigg)^{\frac{4}{m+2}}\\
        &\leq2^{\frac{4}{m+2}}\bigg(\int_X|H_1(x,y,t)|^{\frac{2m}{m-2}}dV_g(y)\bigg)^{\frac{m-2}{m+2}}.\\
    \end{aligned}
\end{equation*}
Combining this with \eqref{equation:strong-Sobolev-ineq} and \eqref{equation:dt-of-H=-2-int}, we get
$$\partial_tH_1(x,x,2t)\leq -C(m,C_S)H_1(x,x,2t)^{\frac{m+2}{m}},$$
so that
$$\partial_t(H_1(x,x,t)^{-\frac{2}{m}})\geq C(m,C_S).$$
Integrating this from $\epsilon$ to $t$ and using the asymptotic behavior $H_1(x,x,\epsilon)^{-\frac{2}{m}}\rightarrow0$ as $\epsilon\rightarrow0$, we arrive at
$$H_1(x,x,t)\leq\frac{C(m,C_S)}{t^{m/2}},\ t>0.$$
Now using the fact $\big|H_1(x,y,t)\big|\leq H_1(x,x,t)^{\frac{1}{2}}H_1(y,y,t)^{\frac{1}{2}}$, we obtain
$$\bigg|H(x,y,t)-\frac{1}{V}\bigg|\leq\frac{C(m,C_S)}{t^{m/2}},\ t>0.$$
\end{proof}

\end{document}